\documentclass[a4paper,11pt]{amsproc}
\usepackage[T1]{fontenc}
\usepackage{lmodern}
\usepackage{amssymb,mathrsfs}
\usepackage{microtype}
\usepackage{needspace}
\usepackage[a4paper,left=22mm,right=22mm,top=20mm,bottom=22mm,
 includehead,headsep=6mm]{geometry}
\usepackage{etoolbox}
\makeatletter
\patchcmd{\@maketitle}{\global\topskip8pc\relax}
 {\global\topskip2pc\relax}{}
 {\PackageError{commutators}{Could not adjust title spacing}{}}
\makeatother
\usepackage[hidelinks]{hyperref}
\hypersetup{pdftitle={Commutators in continuous regular rings},
 pdfauthor={Andreas Thom and Jiaqi Wang}}
\newtheorem{theorem}{Theorem}[section]

\newtheorem{lemma}[theorem]{Lemma}
\newtheorem{corollary}[theorem]{Corollary}
\newtheorem{maintheorem}{Theorem}

\newtheorem{maincorollary}[maintheorem]{Corollary}
\theoremstyle{definition}

\theoremstyle{remark}
\newtheorem{remark}[theorem]{Remark}
\newcommand{\K}{\mathbb K}
\newcommand{\C}{\mathbb C}
\newcommand{\U}{\mathscr U}
\newcommand{\M}{\mathcal M}
\newcommand{\N}{\mathcal N}
\newcommand{\R}{\mathcal R}
\newcommand{\cA}{\mathcal A}
\newcommand{\W}{\mathsf W}
\newcommand{\QW}{\mathsf D}
\DeclareMathOperator{\rank}{rank}
\DeclareMathOperator{\diag}{diag}

\DeclareMathOperator{\ann}{ann}

\dedicatory{To the memory of Richard V. Kadison}
\title{Commutators in continuous regular rings}
\author{Andreas Thom}
\address{A.T., Institute of Geometry, TU Dresden, 01062 Dresden, Germany}
\email{andreas.thom@tu-dresden.de}
\author{Jiaqi Wang}
\address{J.W., State Key Laboratory of Mathematical Sciences, Academy of Mathematics
and Systems Science, Chinese Academy of Sciences and
University of Chinese Academy of Sciences, Beijing 100049, China}
\email{jiaqiwang@amss.ac.cn}
\date{September 25, 2026}
\subjclass[2020]{Primary 16E50; Secondary 16S32, 16U70, 46L10, 47B47}
\keywords{Continuous regular ring, commutator, rank metric, affiliated operator,
Weyl algebra}

\begin{document}
\begin{abstract}
Let $R$ be a continuous von Neumann regular ring with no nonzero
abelian idempotents. We prove that every element of $R$ is a single
additive commutator. In particular, every operator affiliated with a
type $\mathrm{II}_1$ von Neumann algebra is a commutator of operators
affiliated with the same algebra. Our commutator theorem, together with
a theorem of Jaikin-Zapirain and L\'opez-\'Alvarez, implies that, if $R$ is a
nonzero unital algebra over a field $\K$ of characteristic zero, the
quotient division ring of the first Weyl algebra over $\K$ embeds
unitally into $R$. We also give a direct proof of this consequence.
\end{abstract}
\maketitle

\section{Introduction}

The problem of representing the identity as an additive commutator originates in
the operator formulation of the Heisenberg relation.  The trace excludes
such a representation in a complex matrix algebra. The classical
theorems of Wintner \cite{Wintner} and Wielandt \cite{Wielandt} exclude it
for bounded operators on any Hilbert space. 

For a finite von Neumann algebra $\M$, the closed densely defined
operators affiliated with $\M$ form the Murray--von Neumann algebra
$\U(\M)$. Here affiliation means that the operator commutes,
including its domain, with every unitary in the commutant $\M'$.
Addition and multiplication are defined by taking closures.
Thus the equation $[A,B]=AB-BA=1$ is an algebraic question in a unital
$*$-algebra. Kadison and Liu \cite[Section~7, p.~39]{KL} explicitly asked
whether it can hold when the affiliated operators are not required to
be self-adjoint. Kadison, Liu and the first author
\cite[Theorem~1]{KLT} proved that the identity is a sum of two
commutators in $\U(\M)$ whenever $\M$ is of type $\mathrm{II}_1$,
and reiterated the question about a single commutator
\cite[pp.~235, 238]{KLT}.
The late Richard Kadison drew attention to this problem under the name \emph{Heisenberg–von Neumann puzzle.}

For bounded elements, Fack and de la Harpe \cite[Theorem~3.2]{FdH} proved that
the finite sums of commutators in a finite von Neumann algebra are
exactly the elements with vanishing center-valued trace.
Dykema and Kalton \cite[Lemma~4.3]{DK} give twelve commutators for
normal affiliated operators in a type~$\mathrm{II}_1$ factor with
separable predual. Decomposition into real and imaginary parts
therefore expresses every operator affiliated with such a factor as
a sum of at most twenty-four commutators.

\medskip

We work in the more general setting of von Neumann's continuous
regular rings \cite{vN}. All rings are associative and unital. A ring $R$ is
\emph{von Neumann regular} if for every $a\in R$ there is $b\in R$
with $aba=a$. Its principal right ideals form a complemented modular
lattice $\mathcal L(R)$. A regular ring is \emph{continuous} if this
lattice is complete and satisfies the continuity identities for
directed joins and filtered meets. It is \emph{irreducible} if it
is nonzero and is not a direct product of two nonzero rings.
An irreducible continuous ring has a field as its center and a
canonical normalized rank function $\rho:R\to[0,1]$. Such a ring is
called \emph{nondiscrete} if the metric $d(a,b)=\rho(a-b)$ is nondiscrete.
An idempotent $e\in R$ is \emph{abelian} if every idempotent of
$eRe$ is central in $eRe$. Nondiscrete irreducible continuous regular
rings have no nonzero abelian idempotents.
Section~\ref{sec:background} recalls the dimension theory used below.

\medskip

Our main result is the following theorem:
\begin{maintheorem}\label{thm:main}
Every element of a continuous regular ring with no nonzero abelian
idempotents is an additive commutator. In particular, $1 \in R$ is an additive commutator.
\end{maintheorem}

For a field $\K$, the \emph{first Weyl algebra}
$A_1(\K)=\K\langle x,y\rangle/(yx-xy-1)$
is the unital associative $\K$-algebra generated by $x$ and $y$
subject to the relation $yx-xy=1$. It is a left and right Noetherian
domain, hence a left and right Ore domain; see
\cite[Theorems~1.2.9, 1.3.2 and~2.1.15]{MR}. It therefore has a
\emph{quotient division ring}, denoted by $Q(A_1(\K))$, obtained
by inverting all its nonzero elements.

The identity case of Theorem~\ref{thm:main}, together with the
uniqueness theorem for Sylvester matrix rank functions of
Jaikin-Zapirain and L\'opez-\'Alvarez \cite[Proposition~4.2]{JL},
implies the following corollary: apply these results in the standard
continuous ring $\mathfrak C_\K$ of Section~\ref{sec:background},
as in Remark~\ref{rem:weyl-extension}, and then use
Lemma~\ref{lem:embedding}.
We give a direct proof in Section~\ref{sec:weyl}.

\begin{maincorollary}\label{thm:division}
Let $\K$ be a field of characteristic zero, and let $R$ be a nonzero
continuous regular unital $\K$-algebra with no nonzero abelian
idempotents. The division ring $Q(A_1(\K))$ embeds unitally into $R$
as a $\K$-algebra.
\end{maincorollary}

To prove Theorem~\ref{thm:main}, we construct an identity commutator
from integer matrices. On
$\mathbb Z[t]_{<n}$, truncated multiplication $X_n$ and
differentiation $D_n$ satisfy $[D_n,X_n]=I_n-nE_n$, where $E_n$
factors through $\mathbb Z$. Taking remainders modulo $t^n$ and
$(t-1)^n$ compares the model of size $2n$ with two copies of the
model of size $n$, with each error factoring through $\mathbb Z^2$.
Recursive changes of basis yield convergent sequences in compatible
dyadic matrix decompositions of $R$. Convergence is controlled by
the center-valued dimension of principal right ideals.

Section~\ref{sec:commutators} uses this construction to represent
an arbitrary element as a commutator. The main technical lemma constructs a maximal orthogonal family
of idempotents with zero compression that leaves a central compression
in the complementary corner. We regroup central portions of this
family into countably many idempotents. Companion matrices of
pairwise comaximal integer polynomials in the zero-compression
corners, together with an identity commutator in the remaining
corner, give solvable rectangular commutator equations. The
dimension of each column or row correction is bounded by that of
one idempotent, which ensures convergence of the resulting series.

Under the hypotheses of Corollary~\ref{thm:division}, we work in the
rank completion of $\K\to M_2(\K)\to M_4(\K)\to\cdots$.
Lemma~\ref{lem:embedding} embeds this completion into $R$ using
the standard matrix-unit construction; compare Schneider and the
first author \cite{ST}. In characteristic zero, a nonzero polynomial
differential operator has a kernel of bounded dimension on the
finite models. Its limit is therefore invertible.

For affiliated operators our results have the following consequences.

\Needspace{7\baselineskip}
\begin{maincorollary}\label{cor:aff-intro}
If $\M$ is a von Neumann algebra of type $\mathrm{II}_1$, then
every element of $\U(\M)$ is a commutator, and $Q(A_1(\C))$
embeds unitally into $\U(\M)$.
\end{maincorollary}

The identity commutator and the division-ring embedding can both
be constructed in $\U(\R)$ for a unital hyperfinite
$\mathrm{II}_1$ subfactor $\R\subseteq\M$, without a separability
assumption on $\M$.

The commutator factors representing $1$ in $\U(\M)$ necessarily
have substantial analytic irregularity. Ber, Sukochev and Zanin
\cite{BSZ} proved that neither factor can be normal. Further results on the Heisenberg
relation for locally measurable operators were obtained by Ber,
Huang, Kudaybergenov and Sukochev
\cite[Corollaries~3.6 and~3.7]{BHKS}.
In a type $\mathrm{II}_1$ factor, Nayak \cite[Corollary~5.3]{Nayak}
showed that at least one factor fails logarithmic integrability, and
hence belongs to none of the noncommutative $L^p$ spaces,
$0<p\leq\infty$.

\section{Continuous rings and dimension}\label{sec:background}

Unless explicitly stated otherwise, $R$ is a nonzero continuous regular
ring with no nonzero abelian idempotents. Such a ring is directly finite;
see \cite[Corollary~13.23 and Proposition~5.2]{Goodearl}. It is both
right and left self-injective by \cite[Theorem~3 and its corollary]{Utumi}.
Every nonzero corner has the same properties.
We use the dimension theory of these rings as developed in
\cite[Chapters~10--11]{Goodearl} and \cite[Section~5-3]{GW}.

The central idempotents form a complete Boolean algebra. Its Stone
space $\Omega$ is extremally disconnected, and the ordered vector
lattice $\mathcal E=C(\Omega,\mathbb R)$ is Dedekind complete.
We identify a central idempotent with the characteristic function
of its corresponding clopen subset of $\Omega$. The normalized
dimension of principal right ideals takes values in
$C(\Omega,[0,1])$; see \cite[Theorem~3-9.10 and Proposition~5-3.3]{GW}.
Write $\Delta(a)=\Delta(aR)$ for $a\in R$.
For $a,b\in R$,
\begin{gather}
 \Delta(1)=1,\qquad \Delta(a)=0\ \Longleftrightarrow\ a=0,
 \label{eq:rank-basic}\\
 \Delta(a+b)\leq\Delta(a)+\Delta(b),\qquad
 \Delta(ab)\leq\min\{\Delta(a),\Delta(b)\}.
 \label{eq:rank-ineq}
\end{gather}
Dimension is invariant under module isomorphisms and additive on
direct sums. In particular, it is additive on orthogonal idempotents,
and
\begin{equation}\label{eq:rank-nullity}
 \Delta\bigl(\ann_r(a)\bigr)=1-\Delta(a).
\end{equation}
For every central idempotent $z$, one has $\Delta(za)=z\Delta(a)$.
All inequalities and suprema of dimension functions are understood
in the order of $\mathcal E$.

Arbitrary intersections of principal right ideals are principal.
The join of a family of principal right ideals is the principal
right ideal in which their sum is essential; see
\cite[Proposition~5-3.3]{GW}. Here a submodule is \emph{essential}
if it meets every nonzero submodule nontrivially. Dimension preserves
increasing joins and decreasing meets:
\begin{equation}\label{eq:dimension-normal}
 \Delta\left(\bigvee_n I_n\right)
   =\sup_n\Delta(I_n),\qquad
 \Delta\left(\bigcap_n J_n\right)
   =\inf_n\Delta(J_n),
\end{equation}
for increasing sequences $(I_n)$ and decreasing sequences $(J_n)$;
see \cite[Theorem~13.18]{Wehrung}.

For idempotents, $f\leq e$ means $ef=fe=f$. If
$0\leq d\leq\Delta(e)$ in $\mathcal E$, there is an idempotent
$f\leq e$ with $\Delta(f)=d$. Idempotents of equal dimension
are equivalent: if $\Delta(e)=\Delta(f)$, there are
$u\in fRe$ and $v\in eRf$ with $vu=e$ and $uv=f$.
Consequently, for every positive integer $n$, each nonzero corner
$eRe$ admits a unital $n\times n$ matrix decomposition, obtained by
splitting $e$ into $n$ equivalent orthogonal idempotents.

The \emph{central support} $c(e)$ of an idempotent $e$ is the least
central idempotent $z$ with $ze=e$. For a nonnegative function
$d\in\mathcal E$, its \emph{carrier} $\operatorname{car}(d)$ is the central idempotent
corresponding to the clopen set $\overline{\{\omega:d(\omega)>0\}}$.
The carrier of $\Delta(e)$ is $c(e)$.

The dimension inequalities imply
\begin{align}
 |\Delta(a)-\Delta(b)|&\leq\Delta(a-b),\label{eq:rank-cont}\\
 \Delta(a_1b_1-a_2b_2)&\leq\Delta(a_1-a_2)+\Delta(b_1-b_2).
 \label{eq:product-cont}
\end{align}
We say that $a_n$ converges to $a$ in dimension if there are
$t_n\in\mathcal E$ with $t_n\downarrow0$ and
$\Delta(a_n-a)\leq t_n$. Such limits are unique. Addition,
multiplication and evaluation of polynomials with central
coefficients preserve this convergence, by
\eqref{eq:rank-ineq}--\eqref{eq:product-cont}.

\begin{lemma}\label{lem:order-complete}
Let $(a_n)$ be a sequence in $R$, and let
$\varepsilon_n\in\mathcal E$ be nonnegative functions whose
partial sums are bounded above. If
$\Delta(a_{n+1}-a_n)\leq\varepsilon_n$ for every $n$, there is
$a\in R$ such that
\[
 \Delta(a-a_n)\leq t_n,\qquad
 t_n=\sum_{k\geq n}\varepsilon_k\downarrow0.
\]
The sums of functions are order suprema of their finite partial sums.
\end{lemma}

\begin{proof}
Put $I_n=\bigcap_{k\geq n}\ann_r(a_{k+1}-a_k)$. These are
principal right ideals, and $I_n\subseteq I_{n+1}$.
Additivity and \eqref{eq:rank-nullity} give
\[
 \Delta\left(\bigcap_{k=n}^m\ann_r(a_{k+1}-a_k)\right)
 \geq1-\sum_{k=n}^m\varepsilon_k.
\]
Taking the filtered meet and using \eqref{eq:dimension-normal}
yields $\Delta(I_n)\geq1-t_n$. Since $t_n\downarrow0$,
\eqref{eq:dimension-normal} gives $\bigvee_n I_n=R$.
Thus $I=\bigcup_n I_n$ is an essential right ideal.

The maps $I_n\to R$, $x\mapsto a_nx$, agree on their common
domains: if $x\in I_n$ and $m\geq n$, then $a_mx=a_nx$.
They therefore define a right-module map $I\to R$.
By right self-injectivity, it extends to an endomorphism of $R_R$,
which is left multiplication by some $a\in R$.
Since $(a-a_n)I_n=0$, \eqref{eq:rank-nullity} gives
\[
 \Delta(a-a_n)\leq1-\Delta(I_n)\leq t_n.
\]
\end{proof}

In particular, $\sum_j a_j$ converges in dimension whenever the
partial sums of $\sum_j\Delta(a_j)$ are bounded above in
$\mathcal E$, and
\begin{equation}\label{eq:tail}
 \Delta\left(\sum_{j>N}a_j\right)
 \leq\sum_{j>N}\Delta(a_j)\downarrow0.
\end{equation}
This applies whenever $a_j\in e_jRe_j$, or merely
$\Delta(a_j)\leq\Delta(e_j)$, for pairwise orthogonal idempotents
$e_j$.

We shall also use the standard dimension criterion for invertibility
in its center-valued form; compare \cite[Remark~4.2(B)]{BS}.

\begin{lemma}\label{lem:unit}
Let $e\in R$ be a nonzero idempotent. An element $a\in eRe$ is a
unit in $eRe$ if and only if $\Delta(a)=\Delta(e)$.
\end{lemma}

\begin{lemma}\label{lem:central-mixing}
Let $(z_i)_{i\in I}$ be pairwise orthogonal central idempotents,
put $z=\bigvee_i z_i$, and let $x_i\in z_iR$.
There is a unique $x\in zR$ such that $z_ix=x_i$ for all $i$.
\end{lemma}

\begin{proof}
The right ideal $J=\bigoplus_i z_iR$ is essential in $zR$.
Indeed, if $yR\subseteq zR$ is disjoint from $J$, then
$yz_i=0$ for every $i$, so $z_i\Delta(y)=0$ for every $i$.
Since $\bigvee_i z_i=z$ and $y=zy$, faithfulness gives $y=0$.
The map $J\to zR$ whose restriction to $z_iR$ is left
multiplication by $x_i$ extends to an endomorphism of $zR$ by
self-injectivity. This endomorphism is left multiplication by an
$x\in zR$, and $z_ix=xz_i=x_i$. Uniqueness follows because $R$
is nonsingular and $J$ is essential in $zR$.
\end{proof}

For a field $\K$, let
\begin{equation}\label{eq:tower}
 \cA_\K=\varinjlim_k M_{2^k}(\K),\qquad T\longmapsto\diag(T,T),
 \qquad \mathfrak C_\K=\widehat{\cA_\K}^{\,\rho},
\end{equation}
where $\rho(T)=2^{-k}\rank(T)$ on $M_{2^k}(\K)$.
The embeddings preserve rank, and the rank inequalities extend
the ring operations and rank to the metric completion
$\mathfrak C_\K$. This is a nondiscrete irreducible continuous
regular ring, called the standard continuous regular ring over
$\K$; see also \cite[Section~1]{Elek}.
For finite fields, this ring and its unit group are studied by
Carderi and the first author \cite{CT}.

\begin{lemma}\label{lem:embedding}
If $R$ is a unital $\K$-algebra for a field $\K$, there is a unital
$\K$-algebra embedding $\iota:\mathfrak C_\K\hookrightarrow R$
satisfying $\Delta(\iota(T))=\rho(T)1$ for every
$T\in\mathfrak C_\K$.
\end{lemma}

\begin{proof}
Using the matrix decompositions above, the standard matrix-unit argument
gives compatible unital embeddings of $M_{2^k}(\K)$ into $R$,
with dimension equal to normalized matrix rank; compare \cite{ST}.
Applying Lemma~\ref{lem:order-complete} to
rapidly Cauchy subsequences extends their union to the rank
completion, and \eqref{eq:rank-cont} preserves
the stated dimension identity, which also proves injectivity.
\end{proof}

\section{An identity commutator}\label{sec:identity}

For $n\geq1$, let $V_n=\mathbb Z[t]_{<n}$, with its ordered
monomial basis, and define endomorphisms
\begin{equation}\label{eq:models}
 D_nf=f',\qquad X_nf=tf-([t^{n-1}]f)t^n,\qquad
 E_nf=([t^{n-1}]f)t^{n-1}.
\end{equation}
Here $[t^j]f$ denotes a coefficient. Evaluation on the monomial
basis gives
\begin{equation}\label{eq:defect}
 [D_n,X_n]=I_n-nE_n.
\end{equation}
The map $E_n$ factors through $\mathbb Z$.

\begin{lemma}\label{lem:comparison}
For $a=0,1$, let $R_{a,n}f$ be the remainder of $f$ modulo
$(t-a)^n$, regarded as a polynomial in $t$. The map
\[
 \Psi_n:V_{2n}\longrightarrow V_n\oplus V_n,\qquad
 f\longmapsto(R_{0,n}f,R_{1,n}f)
\]
is an isomorphism of free abelian groups. Both maps
\begin{equation}\label{eq:comparison}
 \Psi_nD_{2n}-(D_n\oplus D_n)\Psi_n,
 \qquad \Psi_nX_{2n}-(X_n\oplus X_n)\Psi_n
\end{equation}
factor through $\mathbb Z^2$.
\end{lemma}

\begin{proof}
The ideals $(t^n)$ and $((t-1)^n)$ are comaximal in
$\mathbb Z[t]$, since $(t)$ and $(t-1)$ are comaximal. The
Chinese remainder theorem, with the monic polynomial
$t^n(t-1)^n$, therefore proves that $\Psi_n$ is invertible over
$\mathbb Z$.

Write $f=r_a+(t-a)^nq_a$, where $r_a=R_{a,n}f$.
Differentiation and reduction modulo $(t-a)^n$ give
\[
 R_{a,n}(f')-r_a'=nq_a(a)(t-a)^{n-1}.
\]
Each component of the first map in \eqref{eq:comparison} thus
factors through $\mathbb Z$.

For multiplication, the first component satisfies
$R_{0,n}X_{2n}=X_nR_{0,n}$. Put
$\alpha=[t^{2n-1}]f$ and $\beta=[t^{n-1}]R_{1,n}f$.
Since
\[
 R_{1,n}(tR_{1,n}f)=tR_{1,n}f-\beta(t-1)^n,
\]
the second component is
\[
 R_{1,n}X_{2n}f-X_nR_{1,n}f
 =\beta\bigl(t^n-(t-1)^n\bigr)
   -\alpha R_{1,n}(t^{2n}).
\]
It factors through the map $f\mapsto(\alpha,\beta)$ to
$\mathbb Z^2$.
\end{proof}

\begin{theorem}\label{thm:identity}
There are $A,B\in R$ with $[A,B]=1$ such that $B$ is a unit
and $h(B)$ is a unit for every $h\in\mathbb Z[t]$ with $h(0)=1$.
\end{theorem}

\begin{proof}
Choose compatible dyadic matrix units in $R$, with the inclusions
$\iota_n(T)=\diag(T,T)$ for integer matrices. In a unital
$n\times n$ matrix decomposition, each diagonal matrix unit has
dimension $n^{-1}\mathbf1$. An integer matrix that factors
through $\mathbb Z^r$ therefore has dimension at most
$(r/n)\mathbf1$ when interpreted in these matrix units.

Write $\Phi_n=\Psi_n^{-1}$. Set $n_k=2^k$, $C_0=(1)$, and,
in the monomial coordinates, define
\[
 C_{k+1}=\Phi_{n_k}\iota_{n_k}(C_k),\qquad
 A_k=C_k^{-1}D_{n_k}C_k,\qquad
 B_k=C_k^{-1}X_{n_k}C_k.
\]
The matrices $C_k$ are invertible over $\mathbb Z$, so their
images in $R$ are units. With $n=n_k$,
\[
 A_{k+1}-\iota_n(A_k)
 =\iota_n(C_k)^{-1}
 \bigl(\Phi_n^{-1}D_{2n}\Phi_n-\iota_n(D_n)\bigr)
 \iota_n(C_k),
\]
and the same identity holds with $B$ and $X$ in place of $A$
and $D$. Lemma~\ref{lem:comparison}, invariance of dimension
under multiplication by units, and \eqref{eq:defect} give
\begin{gather}
 \Delta(A_{k+1}-A_k),\ \Delta(B_{k+1}-B_k)
       \leq2^{-k}\mathbf1,\label{eq:cauchy}\\
 \Delta([A_k,B_k]-1)\leq2^{-k}\mathbf1.
 \label{eq:defect-limit}
\end{gather}
Here all differences use the chosen compatible matrix units.
Lemma~\ref{lem:order-complete} gives limits $A,B\in R$.
Continuity of multiplication and \eqref{eq:defect-limit} imply
$[A,B]=1$.

The matrix $X_{2^k}$ shifts $2^k-1$ equivalent diagonal
summands isomorphically onto the other $2^k-1$ summands, so
\[
 \Delta(B_k)=(1-2^{-k})\mathbf1.
\]
Consequently $\Delta(B)=\mathbf1$, and $B$ is a unit by
Lemma~\ref{lem:unit}. If $h(0)=1$, then $h(B_k)$ is a unit
because $B_k$ is nilpotent. Polynomial continuity gives
$\Delta(h(B))=\mathbf1$, and Lemma~\ref{lem:unit} again
proves invertibility.
\end{proof}

Define monic polynomials in $\mathbb Z[t]$ by
\begin{equation}\label{eq:polynomials}
 p_1(t)=t,\qquad p_{j+1}(t)=1+\prod_{i=1}^{j}p_i(t).
\end{equation}
For $i<j$, one has $p_j\equiv1\pmod{p_i}$, so any two
distinct members generate the unit ideal in $\mathbb Z[t]$.
Moreover, $p_j(0)=1$ for every $j\geq2$.

\begin{corollary}\label{cor:corner-pair}
For every nonzero idempotent $f\in R$ there are $P,Q\in fRf$
such that $[P,Q]=f$ and $p_j(P)$ is a unit in $fRf$ for every
$j\geq1$. Polynomial evaluations use the identity $f$.
\end{corollary}

\begin{proof}
Apply Theorem~\ref{thm:identity} in the corner $fRf$ and set
$P=B$ and $Q=-A$.
\end{proof}

\begin{remark}
If $R$ has prime characteristic $p$, its matrix units give a
unital copy of $M_p(\mathbb F_p)$, in which
$[D_p,X_p]=I_p$.
\end{remark}

\section{Every element is a commutator}\label{sec:commutators}

We first choose orthogonal corners on which the given element has zero
compression, apart from one corner on which its compression is central.

\begin{lemma}\label{lem:zero-compression}
If $a\notin Z(R)$, there is a nonzero idempotent $e\in R$ with $eae=0$.
\end{lemma}

\begin{proof}
There is an idempotent $p$ with $b=(1-p)ap\ne0$. Otherwise, applying
the assumed vanishing to both $p$ and $1-p$ shows that $a$ commutes
with every idempotent. In a unital $2\times2$ matrix decomposition,
every off-diagonal entry is a difference of idempotents, and every
diagonal entry is a product of two off-diagonal entries. Thus $a$
would commute with every element of $R$, a contradiction.

By regularity choose $c$ with $bcb=b$. Set $b'=pc(1-p)$ and $q=b'b$.
Then $bb'b=b$, and $q$ is a nonzero idempotent in $pRp$.
Left multiplication by $b$ is injective on $U=qR$, since
$bu=0$ implies $u=qu=b'bu=0$. If $au\in U$ for $u\in U$, then
$0=(1-p)au=bu$, hence $u=0$. Consequently $U\cap aU=0$.

The finitely generated right ideal $U\oplus aU$ is a direct summand
of $R_R$. Project onto $U$ along $aU$ and a complementary right ideal.
This projection is left multiplication by a nonzero idempotent $e$,
and it satisfies $eae=0$.
\end{proof}

\begin{lemma}\label{lem:boolean-refinement}
Let $(w_\alpha)_{\alpha\in A}$ be a family in a complete Boolean algebra,
and let $m\geq1$. Suppose that the meet of any $m+1$ members with
distinct indices is zero. There are elements $v_{\alpha,k}$,
$1\leq k\leq m$, such that
\[
 w_\alpha=\bigvee_{k=1}^{m}v_{\alpha,k},
\]
the pieces in this decomposition are disjoint, and for each $k$ the
family $(v_{\alpha,k})_{\alpha\in A}$ is disjoint.
\end{lemma}

\begin{proof}
Well-order $A$. Having defined the pieces at indices $\beta<\alpha$,
put
\[
 U_{\alpha,k}=\bigvee_{\beta<\alpha}v_{\beta,k},\qquad
 v_{\alpha,k}=w_\alpha\wedge\neg U_{\alpha,k}
                   \wedge\bigwedge_{\ell<k}U_{\alpha,\ell}.
\]
For fixed $\alpha$, these pieces are disjoint, and each piece is
disjoint from all earlier pieces with the same second index. Their
join leaves only
\[
 w_\alpha\wedge\bigwedge_{k=1}^{m}U_{\alpha,k}.
\]
Finite meets distribute over arbitrary joins in a complete Boolean
algebra. If this remainder were nonzero, there would be indices
$\beta_1,\ldots,\beta_m<\alpha$ with
\[
 w_\alpha\wedge v_{\beta_1,1}\wedge\cdots\wedge v_{\beta_m,m}\ne0.
\]
The indices $\beta_k$ must be distinct, because different pieces
belonging to one index are disjoint. Since $v_{\beta_k,k}\leq
w_{\beta_k}$, this contradicts the hypothesis.
\end{proof}

\begin{lemma}\label{lem:zero-partition}
For every $a\in R$ there is a finite or countable family of pairwise
orthogonal nonzero idempotents $(e_i)_{i\in I}$, possibly empty, such
that, with $g=\sum_i e_i$ and $f=1-g$,
\[
 e_i a e_i=0\quad(i\in I),\qquad faf\in Z(fRf).
\]
The series defining $g$ converges in the sense of
Lemma~\ref{lem:order-complete}.
\end{lemma}

\begin{proof}
By Zorn's lemma choose a maximal orthogonal family $(u_\alpha)_{\alpha\in A}$
of nonzero idempotents satisfying $u_\alpha a u_\alpha=0$.
We regroup central portions of these idempotents into a countable
family.

Put
\[
 z_{\alpha,0}=0,\qquad
 z_{\alpha,n}=\operatorname{car}\bigl((\Delta(u_\alpha)-2^{-n})_+\bigr),
 \qquad w_{\alpha,n}=z_{\alpha,n}(1-z_{\alpha,n-1})\quad(n\geq1).
\]
For each $\alpha$, the $z_{\alpha,n}$ increase to the central support
of $u_\alpha$, so the $w_{\alpha,n}$ form a disjoint central partition
of that support. On $w_{\alpha,n}$ the function $\Delta(u_\alpha)$
is at least $2^{-n}$. Since
\[
 \sum_{\alpha\in F}\Delta(u_\alpha)\leq1
 \qquad\text{for every finite }F\subseteq A,
\]
the meet of any $2^n+1$ distinct members of
$(w_{\alpha,n})_{\alpha\in A}$ is zero.

Apply Lemma~\ref{lem:boolean-refinement}, for each $n$, to write
\[
 w_{\alpha,n}=\bigvee_{k=1}^{2^n}v_{\alpha,n,k},
\]
where the pieces for each $\alpha$ are disjoint and, for fixed $n,k$,
the central idempotents $v_{\alpha,n,k}$ are pairwise orthogonal as
$\alpha$ varies. Lemma~\ref{lem:central-mixing} gives an element
$E_{n,k}$ with
\[
 v_{\alpha,n,k}E_{n,k}=v_{\alpha,n,k}u_\alpha
 \quad(\alpha\in A),
\]
and with zero component on the complement of
$\bigvee_\alpha v_{\alpha,n,k}$. Checking these identities on their
central components shows that $E_{n,k}$ is an idempotent and that
$E_{n,k}aE_{n,k}=0$. All the $E_{n,k}$ are pairwise orthogonal:
components originating from different $u_\alpha$ are orthogonal,
and components originating from the same $u_\alpha$ have disjoint
central supports.

Discard the zero members and enumerate this countable family as
$(e_i)_{i\in I}$. Dimension additivity bounds the finite sums of
$\Delta(e_i)$ by $1$, so Lemma~\ref{lem:order-complete} defines
$g=\sum_i e_i$. Continuity of multiplication gives
\[
 g^2=g,\qquad ge_i=e_i=e_ig.
\]
For every $\alpha$, the identities $gu_\alpha=u_\alpha=u_\alpha g$
hold on each central piece $v_{\alpha,n,k}$, since that piece of
$u_\alpha$ is the corresponding piece of $E_{n,k}$. These pieces
join to the central support of $u_\alpha$, so the identities hold
in $R$. Thus $f=1-g$ is orthogonal to every original $u_\alpha$.

If $faf$ were noncentral in $fRf$, Lemma~\ref{lem:zero-compression}
in that corner would give a nonzero idempotent $e\leq f$ with
$eae=e(faf)e=0$. This would enlarge the original maximal family,
a contradiction. Hence $faf\in Z(fRf)$.
\end{proof}

\begin{lemma}\label{lem:sylvester}
Let $e,f$ be nonzero idempotents, let $S\in eRe$ and $N\in fRf$, and
let $p\in\mathbb Z[t]$ satisfy $p(N)=0$ and $p(S)\in(eRe)^\times$.
Then for every $b\in eRf$ and $c\in fRe$ there are $Z\in eRf$ and
$W\in fRe$ such that
\[
 SZ-ZN=b,\qquad NW-WS=c.
\]
Polynomial evaluations use the identities of the respective corners.
\end{lemma}

\begin{proof}
The divided difference
\[
 H(s,t)=\frac{p(s)-p(t)}{s-t}
\]
is a polynomial in two commuting variables with integer coefficients.
On $eRf$, write $L_S$ and $R_N$ for left multiplication by $S$ and
right multiplication by $N$. These operators commute, and
\[
 (L_S-R_N)H(L_S,R_N)=L_{p(S)}-R_{p(N)}=L_{p(S)}.
\]
Thus
\[
 Z=p(S)^{-1}H(L_S,R_N)(b)
\]
solves the first equation. On $fRe$ the same identity gives
\[
 (L_N-R_S)H(L_N,R_S)=-R_{p(S)},
\]
so
\[
 W=-H(L_N,R_S)(c)\,p(S)^{-1}
\]
solves the second equation. Multiplication by $p(S)^{-1}$ commutes
with the relevant multiplication operators because $p(S)^{-1}$
commutes with $S$.
\end{proof}

\begin{proof}[Proof of Theorem~\ref{thm:main}]
The case $a=0$ is immediate. Fix $a\ne0$ and choose the partition
from Lemma~\ref{lem:zero-partition}. Enumerate $I$ by $1,\ldots,m$
or by the positive integers. All sums below range over $I$; empty
sums are zero. Polynomial evaluations in a corner always use its
own identity. Put $c=faf\in Z(fRf)$.

Let $(p_j)$ be the polynomials in \eqref{eq:polynomials}. In each
nonzero corner $e_jRe_j$, choose a unital matrix decomposition of
size $d_j=\deg p_j$. Using its matrix units, form the companion
matrix $N_j$ of the monic integer polynomial $p_j$. Then
$p_j(N_j)=0$.

If $f\ne0$, Corollary~\ref{cor:corner-pair} supplies $P,Q\in fRf$
such that
\[
 [P,Q]=f,\qquad p_j(P)\in(fRf)^\times\quad(j\geq1).
\]
If $f=0$, put $P=Q=0$ and omit all expressions involving inverses
in $fRf$. Set
\begin{equation}\label{eq:uniform-X}
 X=P+\sum_iN_i.
\end{equation}
The series exists by Lemma~\ref{lem:order-complete}, since
$\Delta(N_i)\leq\Delta(e_i)$ and the finite sums of
$\Delta(e_i)$ are bounded by $1$. The element $X$ commutes with
$f$ and every $e_i$, and its compressions to these corners are
$P$ and $N_i$, respectively.

For each $j\in I$, put $h_j=1-e_j$ and $S_j=h_jXh_j$. Since
$a\ne0$ and $e_jae_j=0$, we have $h_j\ne0$. The element $p_j(S_j)$
is invertible in $h_jRh_j$. Indeed, $p_j(P)$ is invertible when
$f\ne0$, and $p_j(N_i)$ is invertible in $e_iRe_i$ for every
$i\ne j$: the polynomials $p_i$ and $p_j$ are comaximal in
$\mathbb Z[t]$, and a Bezout identity evaluated at $N_i$ gives
the inverse. Consequently
\begin{equation}\label{eq:uniform-inverse}
 p_j(S_j)^{-1}=p_j(P)^{-1}
                  +\sum_{i\ne j}p_j(N_i)^{-1},
\end{equation}
where the first term is omitted if $f=0$. The series on the right
exists, since its $i$th term belongs to $e_iRe_i$ and has dimension
$\Delta(e_i)$. Multiplication on either side by $p_j(S_j)$ gives
$f+\sum_{i\ne j}e_i=h_j$, by orthogonality and continuity.
This proves the asserted invertibility.

The zero compression $e_jae_j=0$ implies $ae_j\in h_jRe_j$.
Applying Lemma~\ref{lem:sylvester} with $S=S_j$, $N=N_j$, and
$p=p_j$, choose $Z_j\in h_jRe_j$ such that
\[
 S_jZ_j-Z_jN_j=ae_j.
\]
Since $X$ is diagonal with respect to $e_j+h_j=1$, this says
\begin{equation}\label{eq:uniform-column}
 [X,Z_j]=ae_j.
\end{equation}
If $f\ne0$, apply the same lemma with $S=P$ and $N=N_j$ to choose
$W_j\in e_jRf$ satisfying
\begin{equation}\label{eq:uniform-row}
 [X,W_j]=N_jW_j-W_jP=e_jaf.
\end{equation}
If $f=0$, set $W_j=0$.

The rectangular supports give the dimension bounds
\[
 \Delta(Z_j)\leq\Delta(e_j),\qquad
 \Delta(W_j)\leq\Delta(e_j).
\]
Hence the element
\begin{equation}\label{eq:uniform-Y}
 Y=cQ+\sum_jZ_j+\sum_jW_j
\end{equation}
is well defined in $R$. More explicitly, let
$\epsilon_n=\sum_{j>n}\Delta(e_j)$, with the sum taken in the
order of $C(\Omega,\mathbb R)$. These tails decrease to zero.
By \eqref{eq:tail}, truncating the series in
\eqref{eq:uniform-X} changes $X$ by dimension at most $\epsilon_n$,
and truncating the two series in \eqref{eq:uniform-Y} changes $Y$
by dimension at most $2\epsilon_n$.

Since $c\in Z(fRf)$, we have $[X,cQ]=c[P,Q]=c=faf$, including
when $f=0$. Using \eqref{eq:uniform-column},
\eqref{eq:uniform-row}, and continuity of multiplication, we obtain
\[
 [X,Y]=faf+\sum_j ae_j+\sum_j e_jaf
       =faf+ag+gaf=a,
\]
because $g+f=1$. This calculation also covers finite families. If
the family is empty, then $f=1$, $c=a$, and the construction reduces
to $[P,aQ]=a$.
\end{proof}

\section{The Weyl division ring in characteristic zero}\label{sec:weyl}

Let $\K$ be a field of characteristic zero. Put
$\W=A_1(\K)=\K\langle x,y\rangle/(yx-xy-1)$ and $\QW=Q(\W)$.
As above, $\QW$ is obtained by inverting all nonzero elements of
$\W$; this is its classical Ore localization, since $\W$ is a
Noetherian domain and hence an Ore domain
\cite[Theorems~1.3.5 and~2.1.15]{MR}. Recall also the unique normal
form $P=\sum_{j=0}^m p_j(x)y^j$, with $p_j\in\K[x]$. These
facts follow from the usual filtration, whose associated graded
ring is the polynomial ring in two variables.

Extend the integer matrices of Section~\ref{sec:identity} to
$\K$. Let $A,B\in\mathfrak C_\K$ be the pair constructed in
Theorem~\ref{thm:identity} using the standard dyadic tower. The relation defines a homomorphism
$\varphi:\W\to\mathfrak C_\K$ by $x\mapsto B$, $y\mapsto A$.
We prove directly that all its nonzero values are units.

\begin{lemma}
\label{lem:weyl-rank}
Let $0\ne P=\sum_{j=0}^m p_j(x)y^j\in\W$, with $p_m\ne0$, and
let $d=\max_j\deg p_j$. On $V_n=\K[t]_{<n}$ set
$P_n=\sum_{j=0}^m p_j(X_n)D_n^j$. Then
\begin{equation}\label{eq:weyl-rank}
 \rank(P_n)\geq n-m-d.
\end{equation}
\end{lemma}

\begin{proof}
On $\K[t]$, consider $L=\sum_{j=0}^m p_j(t)(d/dt)^j$.
Choose $a\in\K$ with $p_m(a)\ne0$, which is possible because
$\K$ is infinite. We claim that a nonzero polynomial in $\ker L$
cannot vanish at $a$ to order $r\geq m$. Indeed, if
$f=c(t-a)^r+\text{higher powers}$, $c\ne0$, the coefficient of
$(t-a)^{r-m}$ in $Lf$ is
\[
 p_m(a)c\,r(r-1)\cdots(r-m+1)\ne0.
\]
All terms with fewer derivatives have higher vanishing order.
For $m=0$ the product is interpreted as $1$, and the same argument
shows that $\ker L=0$. For $m>0$, the map taking the first $m$
Taylor coefficients at $a$ is therefore injective on $\ker L$.
In either case, $\dim_\K\ker L\leq m$.

Let $\pi_n:\K[t]\to V_n$ delete all terms of degree at least $n$.
Since differentiation preserves $V_n$ and
$p_j(X_n)g=\pi_n(p_j(t)g)$ for $g\in V_n$, we have
$P_n=\pi_nL|_{V_n}$. Furthermore $L(V_n)\subseteq V_{n+d}$.
The restriction of $\pi_n$ to $V_{n+d}$ has a kernel of dimension
$d$, whereas $L|_{V_n}$ has a kernel of dimension at most $m$.
This proves \eqref{eq:weyl-rank}.
\end{proof}

\begin{theorem}\label{thm:weyl}
The homomorphism $\varphi$ extends uniquely to a unital embedding
$\QW\hookrightarrow\mathfrak C_\K$. Every nonzero element of
$\QW$ has image of normalized rank one.
\end{theorem}

\begin{proof}
For a nonzero $P$ in the normal form of Lemma~\ref{lem:weyl-rank},
the simultaneous similarities in Theorem~\ref{thm:identity} give
\[
 \sum_{j=0}^m p_j(B_k)A_k^j=C_k^{-1}P_{2^k}C_k.
\]
The expression on the left converges to $\varphi(P)$ in rank.
By \eqref{eq:weyl-rank} and rank continuity,
\[
 \rho(\varphi(P))
 =\lim_{k\to\infty}2^{-k}\rank(P_{2^k})=1.
\]
In particular $\varphi$ is injective and, by Lemma~\ref{lem:unit},
sends every nonzero element to a unit. The universal property of
Ore localization gives the extension to $\QW$. A unital
homomorphism from a division ring is injective. Its nonzero values
are units and hence have rank one.
\end{proof}

Composing with Lemma~\ref{lem:embedding} proves
Corollary~\ref{thm:division}.
A related rank argument followed by Ore localization was used by
Carderi and the first author
\cite[Lemma~4.3 and the proof of Theorem~4.2]{CT} to embed quotient
division rings of certain group algebras into continuous rings over
finite fields.

\begin{remark}
\label{rem:weyl-extension}
Let $S$ be a regular unital $\K$-algebra with a faithful normalized
Sylvester matrix rank function. Every unital $\K$-algebra
homomorphism $\psi:\W\to S$ extends uniquely to an embedding
$\QW\hookrightarrow S$, preserving rectangular matrix ranks.
Indeed, the matrix rank of $S$, normalized by $\rank(I_r)=r$,
pulls back along $\psi$ to a Sylvester matrix rank function on
$\W$. Since $\W$ is simple and left Noetherian,
the uniqueness theorem of Jaikin-Zapirain and L\'opez-\'Alvarez
\cite[Proposition~4.2]{JL} identifies this rank with the rank
induced by $\QW$. Hence every nonzero element of $\W$ has image
of rank one. Such an element $a\in S$ is a unit: choose $b\in S$
with $aba=a$; the idempotents $ab$ and $ba$ have rank one, so
$ab=ba=1$. Ore localization gives the extension, and matrix
reduction over $\QW$ proves the assertion about rectangular ranks.
\end{remark}

\section{Affiliated operators and the Heisenberg relation}
\label{sec:affiliated}

Let $\N\subseteq\mathcal B(\mathcal H)$ be a finite von Neumann
algebra. A closed densely defined operator is \emph{affiliated}
with $\N$ if it commutes, including its domain, with every unitary
in $\N'$. With sums and products defined by closure, these operators
form the unital $*$-algebra $\U(\N)$. For this construction and
its role in commutator questions, see Kadison and Liu
\cite[Sections~6--7, in particular Theorem~6.13]{KL}, and
Kadison, Liu and the first author \cite[Section~2]{KLT}.

If $\N$ has a faithful normal tracial state $\tau$, define
\[
 \rho(T)=\tau(r(T))=\tau(\ell(T)),\qquad T\in\U(\N),
\]
where $r(T)$ and $\ell(T)$ are its right and left support
projections. Their traces agree by polar decomposition. This rank
satisfies the scalar versions of \eqref{eq:rank-basic},
\eqref{eq:rank-ineq}, \eqref{eq:rank-cont} and \eqref{eq:product-cont}.
The rank-completion theorem identifies $\U(\N)$ with the
completion of $\N$ in this metric; see the first author
\cite[Lemma~2.2]{Thom}. Density can also be seen directly from
spectral truncations: if $e_n=1_{[0,n]}(|T|)$, then $Te_n\in\N$
and $\rho(T-Te_n)\leq\tau(1-e_n)\to0$. A unital matrix
$*$-subalgebra $M_m(\C)\subseteq\N$ carries normalized matrix
rank, since $\tau$ restricts to its unique normalized trace.

Under the same trace hypothesis, $\tau$ induces a faithful normalized Sylvester matrix rank
on the regular ring $\U(\N)$. Thus Remark~\ref{rem:weyl-extension}
shows that every unital complex algebra homomorphism
$A_1(\C)\to\U(\N)$ extends uniquely to an embedding of
$Q(A_1(\C))$, preserving rectangular matrix ranks.

We use the classical fact that every von Neumann algebra $\M$ of
type $\mathrm{II}_1$ contains a unital hyperfinite $\mathrm{II}_1$
subfactor $\R$ generated by a compatible dyadic tower of matrix
$*$-subalgebras; see \cite[Exercise~12.4.25 and Theorem~12.2.1]{KR}.

\begin{proof}[Proof of Corollary~\ref{cor:aff-intro}]
Choose such a subfactor $\R\subseteq\M$ and its matrix tower.
The inclusion of this tower into $\U(\R)$ is isometric for
normalized rank and therefore extends to an embedding
$\mathfrak C_\C\hookrightarrow\U(\R)$.
Theorems~\ref{thm:identity} and~\ref{thm:weyl} give the identity
commutator and the division-ring embedding there. The natural
inclusion $\U(\R)\hookrightarrow\U(\M)$ extends the unital
inclusion of finite von Neumann algebras. For example, an affiliated
operator is represented by a fraction $ab^{-1}$ with $a,b\in\R$
and $b$ having zero kernel; the same fraction defines its image
in $\U(\M)$.

By \cite[Proposition~4.4]{Schneider}, $\U(\M)$ is a continuous
regular ring, and $p\mapsto p\U(\M)$ identifies the projection
lattice of $\M$ with its principal-right-ideal lattice.
Given a nonzero idempotent $e\in\U(\M)$, choose a projection
$p\in\M$ with $e\U(\M)=p\U(\M)$. The corners
$e\U(\M)e$ and $p\U(\M)p$ are isomorphic. Since $p\M p$
contains noncommuting projections, $e$ is not abelian.
Thus $\U(\M)$ has no nonzero abelian idempotents, and
Theorem~\ref{thm:main} applies to every element of $\U(\M)$.
\end{proof}

Let $\M$ be a type $\mathrm{II}_1$ factor with normalized trace $\tau$.
If $[A,B]=1$ in $\U(\M)$, the theorem
of Ber, Sukochev and Zanin \cite{BSZ} implies that neither $A$ nor
$B$ is normal. Nayak \cite[Corollaries~5.3 and~5.4]{Nayak} gives
\[
 \tau(\log^+|A|)=\infty\quad\hbox{or}\quad
 \tau(\log^+|B|)=\infty,
\]
and shows that $A-\lambda1$ and $B-\lambda1$ are invertible in
$\U(\M)$ for every $\lambda\in\C$. In particular both factors
have empty point spectrum. For the identity pair obtained from
$\mathfrak C_\C$, the invertibility of all scalar shifts follows
directly from Theorem~\ref{thm:weyl}.

\section*{Acknowledgments}

GPT-6 Astra and GPT-5.6 Sol were used to assist with exploring proof
strategies, developing and checking arguments, searching the literature,
and revising the exposition. The authors take full responsibility for
the mathematical content and the final text.

\end{document}